\documentclass{amsart}
\usepackage[margin=1.5in]{geometry}

\usepackage[all]{xy}
\usepackage{tikz-cd}
\usepackage{enumerate}
\usepackage{amsfonts}
\usepackage{amssymb, amscd}
\usepackage{graphicx}
\usepackage{mathtools}
\usepackage{xcolor}
\usepackage{float}

\newcommand{\sectionbreak}{\vspace{15pt}\centerline{*\quad*\quad *}\vspace{8pt}}

\newcommand{\so}{\mathfrak{so}}

\newcommand{\la}{\langle}
\newcommand{\ra}{\rangle}
\newcommand{\tr}{\mathrm{tr}}
\newcommand{\Ad}{\mathrm{Ad}}

\renewcommand{\H}{\mathbb{H}}
\newcommand{\R}{\mathbb{R}}
\newcommand{\C}{\mathbb{C}}
\newcommand{\N}{\mathbb{N}}
\newcommand{\Z}{\mathbb{Z}}

\newcommand{\G}{G_K^0}
\newcommand{\h}{\mathfrak{h}}
\newcommand{\g}{\mathfrak{g}}
\renewcommand{\t}{\mathfrak{t}}

\newcommand{\n}{\noindent}
\newcommand{\mn}{\medskip\noindent}
\newcommand{\sn}{\smallskip\noindent}
\newcommand{\ms}{\medskip}
\newcommand{\bn}{\bigskip\noindent}

\newtheorem{teo}{Theorem}[section]
\newtheorem{cor}[teo]{Corollary}

\newtheorem{lema}[teo]{Lemma}

\theoremstyle{definition}

\newtheorem{rmk}[teo]{Remark}

\title{On the isometric conjecture of Banach}

\author[G. Bor]{Gil Bor}
\author[L. Hern\'andez]{Luis Hern\'andez Lamoneda}
\address{CIMAT, A.P. 402, Guanajuato, Gto. 36000, Mexico}
\email{gil@cimat.mx,lamoneda@cimat.mx}

\author[V. Jim\'enez]{Valent\'{\i}n Jim\'enez-DeSantiago}
\author[L. Montejano]{Luis Montejano  Peimbert }
\address{IMATE, UNAM, Sede Juriquilla, Queretaro, Mexico}
\email{vale\_jds@hotmail.com,luis@matem.unam.mx}
\date{\today}

\begin{document}
\maketitle
\begin{abstract}
Let $V$ be a  Banach space where for fixed $n$, $1<n<\dim(V)$, all of its  $n$-dimensional subspaces are isometric. In 1932, Banach asked if under this hypothesis $V$ is  necessarily a Hilbert space. Gromov, in 1967, answered it positively for  even $n$  and all $V$. In this paper we give a  positive answer for real  $V$  and odd $n$ of the form  $n=4k+1$, with the possible exception of $n=133.$ Our proof relies on a new characterization of ellipsoids in $\R^n$, $n\geq 5$, as the only symmetric convex bodies all of whose linear hyperplane sections are linearly equivalent affine bodies of revolution.

\end{abstract}

\section{Introduction}

%
%
S. Banach asked in 1932  the following question: 

\sn

{\em Let $V$ be a Banach space, real or complex, finite or infinite dimensional, all of whose $n$-dimensional subspaces, for some fixed integer $n$, $2\leq n<\dim(V)$, are isometrically isomorphic to each other. Is it true that $V$ is a Hilbert space?} (See \cite[p.~244]{B},  or p.~152 of the English translation, remarks on Chap. XII, property (5).)

\ms
%
%

It's important to note that Banach's question\footnote{
Following a long established tradition starting with \cite{G}, we rename  Banach's  question  a `conjecture' in this article, although Banach himself, as far as we know, did not conjecture a positive answer.} is a codimension one problem: since every Banach space, all of whose subspaces of a fixed dimension $n\geq2$ are Hilbert spaces, is itself a Hilbert space\footnote{which easily follows from the elementary characterization of a norm coming from an inner product via the ``parallelogram law".}, an affirmative answer for $n$ in codimension one implies immediately an affirmative answer for $n$ in all codimensions. 

The conjecture  was proved first  for $n=2$  and  real $V$   in 1935 by Auerbach, Mazur and Ulam  \cite{AMU} and for all $n\geq 2$ and infinite dimensional real $V$ in  1959 by A.~Dvoretzky \cite{Dv}. In 1967 M.~Gromov \cite{G} proved  the conjecture for  even $n$ and all $V$, real or complex,  for odd $n$  and  real $V$  with $\dim(V)\geq n+2$, and for odd $n$ and  complex $V$ with $\dim(V)\geq 2n$ (which also proves the conjecture  for all infinite dimensional $V,$ real or complex, as noted above). It is probably worth noticing that V.~Milman \cite{Mi} extended Dvoretzky's theorem to the complex case, in particular reproving Banach's conjecture for infinite dimensional complex $V$.
A recent and very thorough account of the history of this conjecture is found  in \cite[\S6, p.~388]{So}. 
We also recommend \cite{Pe} and the notes on \S9 in \cite[p.~206]{MMO}. 

\smallskip
In this article we settle Banach's conjecture for real $V$ and `one half' of the odd $n$, by showing that: 

\medskip

\mn{\bf Main theorem. }{\em
 A real  Banach space all of whose $n$-dimensional subspaces are isometrically isomorphic to each other for some fixed odd integer $n$ of the form $n=4k+1\geq 5$, $n\neq 133$,  is a Hilbert space.}

\medskip

\begin{rmk} The reason for the strange exception $n\neq 133$ will become clearer during the proof (133 is the dimension of the exceptional Lie group $E_7$). 
\end{rmk}

\medskip

Consider the closed unit ball $B=\{ \|x\|\leq 1\}\subset V$. It is a symmetric convex body. Since a finite dimensional Banach space is a Hilbert space if and only if $B$ is an ellipsoid,  Banach's question  can be reformulated as: 

\ms

{\em Let  $B\subset \R^{N}$ be a symmetric convex body, all of whose sections by $n$-dimensional linear subspaces, for some fixed integer $n$,  $1<n<N$, are linearly equivalent. Is it true that $B$ is an ellipsoid?}
\ms

Thus, in order to prove the Main Theorem, in the sequel we show the following,

\begin{teo}\label{thm:main}  Let $B\subset  \R^{n+1}$,  $n=4k+1\geq 5$, $n\neq 133$, be a convex symmetric body,  all of whose sections by $n$-dimensional  subspaces  are linearly equivalent. Then $B$  is an ellipsoid.
\end{teo}

\medskip

In fact, using Theorem 1 of \cite{Mo1}, one can drop the symmetry assumption on $B$ in the above reformulation, obtaining: 

\mn{\bf Main convex geometry theorem.} {\em  Let $B\subset  \R^{n+1}$,  $n=4k+1\geq 5$, $n\neq 133$, be a convex body,  all of whose sections by $n$-dimensional affine subspaces through a fixed interior point  are affinely equivalent. Then $B$  is an ellipsoid.
}

\subsection{Sketch of the proof of the main theorem}
Our proof of Theorem \ref{thm:main} combines two main ingredients: convex geometry  and  algebraic topology. To describe these, we need to  recall first some standard  definitions.

 A {\em symmetric convex body} is a compact convex subset of a finite dimensional real  vector space  with a nonempty interior, invariant under $x\mapsto -x.$  A {\em hyperplane} is a codimension 1 linear subspace. An {\em affine hyperplane} is the translation of a hyperplane by some vector. A {\em  hyperplane section} of a subset in a vector space  is its intersection with a hyperplane.  Two sets, each a subset of a vector space,   are 
%
{\em linearly} (respectively, {\em affinely}) {\em equivalent} if they can be mapped to each other by a linear (respectively, affine)  isomorphism between their 
%
%
 ambient vector spaces. An {\em ellipsoid} is a subset of a vector space which is affinely  equivalent to the unit ball in euclidean space. 
 
  A symmetric convex body $K\subset\R^n$ is a {\em symmetric body of revolution} if it admits an {\em axis of revolution}, i.e.,    a 1-dimensional linear subspace $L$ such that  each  section of $K$ by an affine hyperplane $A$ orthogonal to $L$ is  an $n-1$ dimensional closed euclidean ball in $A$, centered at $A\cap L$  (possibly empty or just a point). If $L$ is an axis of revolution of $K$ then  $L^\perp$ is the  associated  {\em hyperplane of revolution}. An {\em affine symmetric  body of revolution} is a convex body linearly equivalent to a symmetric body of revolution. The images, under the linear equivalence, of an  axis of revolution and its  associated hyperplane of revolution of the body of revolution are  an axis of revolution  and associated hyperplane of revolution of the affine body of revolution (not necessarily perpendicular anymore). Clearly, an ellipsoid centered at the origin is an affine symmetric body of revolution and any hyperplane serves as a hyperplane of revolution.

With these definitions understood, the  convex geometry result that we use in  the proof of Theorem \ref{thm:main} is the following characterization of ellipsoids. 

\begin{teo}\label{thm:rev}
A symmetric convex body $B\subset \R^{n+1}$,  $n\geq 4$,  all of whose  hyperplane sections are linearly equivalent affine  bodies of revolution,  is an ellipsoid.  
\end{teo}

%
%
The main ingredient in the proof of this theorem is the following result, possibly of independent interest. 

\begin{teo}\label{thm:elip} Let $B\subset \R^{n+1}$, $n\geq 4$,  be a symmetric convex  body, all of whose   hyperplane sections are  affine bodies of revolution. Then, at least one of the sections is an ellipsoid. 
\end{teo}

Note that  in  Theorem \ref{thm:elip}, unlike Theorem \ref{thm:rev}, we do not assume that all hyperplane sections of $B$ are necessarily linearly equivalent to each other. If we add this assumption then it follows from Theorem \ref{thm:elip} that {\em all}  hyperplane sections of $B$ are ellipsoids. 
It then follows easily that {\it $B$ itself is an ellipsoid}: all hyperplane sections are Hilbert spaces and therefore $V$ itself is also one\footnote{In fact, this classical result is known to hold (in every codimension) even without the symmetry assumption on $B$ (see, e.g.,  Theorem 2.12.4 of \cite{MMO} or \cite{So}). It is an open question whether a symmetric convex body all of whose sections are affine symmetric bodies of revolution is itself an affine body of revolution (the converse of Lemma \ref{lema:planes}). 
In Remark \ref{rmk:conv} we briefly discuss this question and explain why Theorem \ref{thm:elip} may be considered as a first step towards an affirmative answer.}.

Theorem \ref{thm:elip} is proved in Section \ref{sec:rev}. The rest of the article consists of topological methods to show that, under the hypotheses of  Theorem \ref{thm:main}, all hyperplane sections of $B$ are necessarily affine  symmetric bodies  of revolution. The link to  topology  is via a  beautiful idea that traces back to the work of Gromov \cite{G}. It consists of the following key observation.

\begin{lema}\label{lema:key} Let $B\subset \R^{n+1}$ be a symmetric convex body, all of whose hyperplane sections are linearly equivalent to some fixed symmetric convex body $K\subset \R^n$. Let $G_K:=\{g\in GL_n(\R)| g(K)=K\}$ be the  {\em group of linear symmetries} of $K$. Then the structure group of  $S^n$ can be reduced to  $G_K$. 
\end{lema}

See Section  \ref{sec:str} below for a proof of this lemma, as well as a brief reminder about structure groups of differentiable manifolds and their reductions. Lemma \ref{lema:key} can be interpreted through the notion of  {\em a field of convex bodies} tangent to $S^n$. See, for example, Mani \cite{Ma} and \cite{Mo1}. 

Following Lemma \ref{lema:key}, our
task is to understand the possible reductions of the structure group of $S^n$ (a classical problem in topology). The results we need are contained in the next purely topological theorem which, when applied to   Lemma \ref{lema:key} with the dimension hypothesis  of Theorem \ref{thm:main}, implies that  $K$ is an affine symmetric body of revolution.

But first, another definition.  We say that a subgroup $G\subset GL_n(\R)$ is {\em reducible} if the induced action on $\R^n$ leaves invariant a $k$-dimensional linear subspace, $1<k<n$; otherwise, it is an {\em irreducible} subgroup of $GL_n(\R)$. 
(Beware of the potentially confusing  use of the notions `reducible' and `can be reduced' in the statement of the following theorem.) 

\begin{teo}\label{thm:red}Let $n\equiv 1$ mod 4, $n\geq 5$,  and suppose that  the structure group of 
$S^n$  can  be reduced  to a closed connected subgroup $G\subset SO_n$. Then: 
\begin{enumerate}[{(a)}]
\item If $G$ is reducible then it is conjugate  to a subgroup of  the standard inclusion $SO_{n-1}\subset SO_n,$ acting transitively on $S^{n-2}$.  
\item  If $G$ is irreducible then $G=SO_n$, or $n=133$ and $G\subset H\subset SO_{133},$ where $H$ is the adjoint representation of the simple exceptional Lie group $E_7$. 
\end{enumerate}
\end{teo}
We prove  Theorem \ref{thm:red} in Section \ref{sec:red} by applying to our situation some  known results from the literature about structure groups on spheres, mainly from \cite{St}, \cite{L} and \cite{CC}. 
Furthermore, for case (b) (the irreducible case), we need to supplement these results with several  facts about  the representation theory and topology of compact Lie groups.  

\sectionbreak

In summary, Theorem \ref{thm:main} is a consequence of  the above results, as follows. 
Since all hyperplane sections of $B$ are linearly equivalent to each other, they are linearly equivalent to some  fixed symmetric convex body $K\subset \R^n$.  
By Lemma \ref{lema:key}, 
the structure group of $S^n$ can be reduced 
to $G_K$. 
It is  easy to see that it can be further reduced to the  identity component $\G\subset G_K$ (Remark \ref{lema:con}). For a convex body $K$, $G_K$ (and thus $\G$) is compact (Lemma \ref{lema:bdd}) and therefore $\G$ is conjugate to a subgroup of $SO_n$
(Lemma \ref{lema:comp}); 
%
%
hence, by passing to a convex body linearly equivalent to $K$,  we can assume that $\G\subset SO_n$. Next, Theorem \ref{thm:red} applied to  $G=\G$, implies that $K$ is a symmetric body of revolution:  in case (a), $\R e_n$ is an axis of revolution  of $K$;  in case (b), $K$ is a euclidean ball. Thus all hyperplane sections of $B$ are linearly equivalent to the  symmetric  body of revolution $K$. It follows, by Theorem \ref{thm:rev}, that $B$ is an ellipsoid. \qed

\sn{\bf Acknowledgments.} We wish to thank Omar Antolin for very helpful conversations, and to Ilia Smilga for kindly contributing  Lemma \ref{lema:ilia}. LM acknowledges  support  from CONACyT under 
project 166306 and  support from PAPIIT-UNAM under project IN112614, whereas GB and LH acknowledge  support  from CONACyT under project 2017-2018-45886.

\section{Affine bodies of Revolution}\label{sec:rev}

The aim of this section is to prove Theorem \ref{thm:rev}, announced in the introduction. For that purpose, we collect here   the following   lemmas.

\subsection{Some preliminary lemmas}
The first two lemmas are quite standard, we give proofs for the convenience of the reader.

\begin{lema}\label{lema:bdd} Let $K\subset \R^n$ be a symmetric  convex body. Then its linear symmetry group $G_K=\{g\in GL_n(\R)\,|\, g(K)=K\}$ is compact. 
\end{lema}
\begin{proof}
Let  $A_K:=\{a\in End(\R^n)\,|\, a(K)\subset K\}$. Since $K$ is closed in $\R^n$, $A_K$  is closed in  $End(\R^n)\simeq\R^{n^2}$   (this follows easily  from the continuity of matrix multiplication $End(\R^n)\times\R^n\to \R^n$). Since $K$ is bounded and $0$ is an interior point, there exist $R,r>0$ such that $B_r\subset K\subset B_R$, where $B_\rho\subset \R^n$  is the closed ball of radius $\rho$. It follows that for every $a\in A_K$, $a(B_r)\subset B_R$, hence $\|a\|\leq R/r$. Thus $A_K\subset End(\R^n)$ is also bounded and hence compact. It remains to show that  $G_K\subset A_K$  is closed. Let $g_i\in G_K$ with $g_i\to g\in End(\R^n).$ 
Since $(g_i)^{-1}\in A_K$,  $(g_i)^{-1}(B_r)\subset B_R$, hence $0<(r/R)\|v\|\leq \|g_iv\|$ for all $i$ and all $v\neq 0$. Taking $i\to\infty$ we get $0<(r/R)\|v\|\leq \|gv\|$, hence $g$ is invertible, i.e., $g\in G_K$. 
\end{proof}

\begin{lema}\label{lema:comp}Every compact subgroup $G\subset GL_n(\R)$  is conjugate to a subgroup of $O_n$.
\end{lema}
\begin{proof}By taking an arbitrary positive inner product on $\R^n$ (e.g., the standard inner product $\sum x_i y_i$) and averaging it over $G$ with respect to a bi-invariant measure, one obtains a $G$-invariant inner product $\la \ ,\ \ra$ on $\R^n$. Now any two inner products on $\R^n$ are linearly isomorphic  to each other, hence one can find an element $g\in GL_n(\R)$ such that $(u,v)\mapsto \la gu, gv\ra$ is the standard inner product on $\R^n$. It follows that $g^{-1}Gg\subset O_n$. For more details see, e.g.,  Prop.~3.1 on p.~36 of \cite{A}. 
\end{proof}
There is also an alternative geometric proof of Lemma \ref{lema:comp} via the notion of minimal ellipsoids, as in \cite[Lemma 1]{G}.

\begin{lema}\label{lema:two}
A symmetric  affine  body of revolution $K\subset \R^n$, $n\geq 3$, admitting two different  hyperplanes of revolution, is an  ellipsoid.
\end{lema}

\begin{proof} Let $G_K=\{g\in GL_n(\R)\,|\, g(K)=K\}$ and 
let $G=\G$  be the identity component of  $G_K$. By Lemma \ref{lema:comp}, $G$  is conjugate to a subgroup of $SO_n$, we may assume, by passing to a body of revolution linearly equivalent to $K$, that $G\subset SO_n$. 
We will show that in this case $K$ is a ball centered at the origin, by showing that $G=SO_n$. 

Now, each hyperplane of revolution of $K$ gives rise to a subgroup of $G$  conjugate in $SO_{n}$ to $SO_{n-1}$ (the stabilizer of the hyperplane). Thus, our hypotheses imply that $SO_{n-1}\subsetneq G\subset SO_n$. But it is well known that $SO_{n-1}$ is a {\it maximal connected} subgroup of $SO_n$, i.e. $G=SO_n$ (see \cite[Lemma 4]{MS}, p. 463).

\end{proof}

\begin{lema}\label{lema:planes}
 Let $K\subset\R^n$, $n\geq 3$,  be an affine symmetric body of revolution. Then any section $K'=\Gamma\cap K$ with  a $k$-dimensional linear subspace $\Gamma\subset \R^n$,  $1<k<n$,  is an affine  symmetric body of revolution in $\Gamma$. Furthermore, if  $L$ is an axis of revolution of $K$ and $H$ the associated   hyperplane of revolution then 

\begin{enumerate}[(a)]
  
\item  If $\Gamma\subset H$ then $K'$ is an ellipsoid. 

\item  If $\Gamma\not\subset H$ then  $H':=\Gamma\cap H$ is a  hyperplane of revolution of $K'$.

\item  If  $L\subset \Gamma$  then $L$ is also the axis of revolution of $K'$ associated to the hyperplane of revolution $\Gamma\cap H$. 

\end{enumerate}
\end{lema}

\begin{proof}(a) If $\Gamma\subset H$ then $\Gamma\cap K$ is a linear section of the ellipsoid 
$H\cap K$, hence is an ellipsoid. 

\sn (b) We can assume,  by applying an appropriate  linear transformation,   
 that $K$ is a symmetric body of revolution with an axis of revolution $L=\R e_n$
  and plane of revolution $H=L^\perp=\{x_n=0\}$, such that $H\cap K$ is the unit ball in $H$ and  $H\pm e_n$ are support hyperplanes of $K$ at 
  $\pm e_n$. Furthermore, we can also arrange that $H':=\Gamma\cap H$ is spanned 
  by  $e_1,\ldots, e_{k-1}$ and so $\Gamma$ is spanned  by  $e_1,\ldots, e_{k-1}, v$, where $v=\lambda e_{n-1}+e_n$ for some $\lambda\in \R$. 
  To show that $H'$ is a hyperplane of revolution of $K'$ with an associated 
  axis of revolution  $L'=\R v$, we need to show that every non empty  
  section of $K'$ by an affine hyperplane of the form $H'+tv$, $t\in\R$, 
  is an $(n-2)$-dimensional  ball in  $H'+tv$, centered at $tv$. 
  The latter section is the section of the $(n-1)$-dimensional   
  ball $(H+te_n)\cap K$, centered at $te_n$,  by $H'+tv$, an 
  affine hyperplane of $H+te_n$, hence is an $(n-2)$-dimensional 
   ball, centered at $tv$, as needed.
   
   \sn(c) In the previous item, if $L\subset \Gamma$, we can choose $v=e_n$. 
\end{proof}

\begin{lema}\label{lema:axis} Let $K\subset\R^n$, $n\geq 3$,  be an affine  symmetric  body of revolution with an axis of revolution $L$. Suppose a section of $K$ by a linear subspace $\Gamma\subset \R^n$ of dimension $\geq 2$ passing through $L$ is an ellipsoid. Then $K$ is an ellipsoid. 
\end{lema}

\begin{proof}Let $e_1, \ldots, e_{n}$ be the standard basis of $\R^{n}$.  By passing to a linearly equivalent  body of revolution, we can assume that $K$ is a symmetric body of revolution with an 
%
%
axis of revolution $L=\R e_{n}$ and associated hyperplane of revolution $H=L^\perp=\{x_{n}=0\}$. Furthermore, we can also assume that $H\cap K$ is the unit ball in $H$ and that $H\pm e_{n}$ are support hyperplanes of $K$ at $\pm e_{n}$. 
%
%
We will show that, under these assumptions,  $K$ is the unit ball in $\R^{n}$. To this end, it is enough to show that each section of $K$ by a 2 dimensional subspace  $\Delta$ containing $L$ is the  unit disk in $\Delta$ centered at the origin.  Let us choose a 2-dimensional subspace $\Delta\subset \Gamma$ containing $L$  and  a unit vector $v$ in the 1-dimensional space $\Delta\cap H$. Then $\Delta\cap K$ is a (solid) ellipse, centered at the origin, whose boundary passes through $\pm v, \pm e_{n}$, with support lines $\R v\pm e_{n}$ at $\pm e_{n}$. It follows that $\Delta\cap K$ is the unit disk  in $\Delta$ centered at the origin. Now since $L=\R e_{n}$ is an axis of revolution of $K$, all rotations in $\R^{n}$ about $L$ leave $K$ invariant. Applying all such rotations to $\Delta$, we obtain all 2-dimensional subspaces containing $L$, and each of them intersects $K$ in a unit disk centered at the origin, as needed.  
\end{proof}

\begin{lema}\label{lema:k1}
Let  $B\subset \R^{n+1}$ be a symmetric convex body, $n\geq 4$, $\Gamma_1, \Gamma_2\subset\R^{n+1}$ two distinct  hyperplanes, such that the hyperplane sections  $K_i:=\Gamma_i\cap B$, $i=1,2$,  are affine symmetric bodies of revolution, 
with  axes and associated hyperplanes of revolution $L_i,H_i$ 
(respectively). 
%
%
If $L_1\subset H_2$ then $K_1$ is an ellipsoid. 
\end{lema}

\begin{proof}Let $ E:=K_1\cap K_2.$
 We will show that $E$ is an ellipsoid. This implies, by Lemma \ref{lema:axis},  that $K_1$ is an ellipsoid, since  $ E=K_1\cap \Gamma_2$ and $\Gamma_2$ contains $L_1$, an axis of revolution of $K_1.$ 

To show that $E$ is an ellipsoid, we note first that $\Gamma_2$ does not contain $H_1$, else  $L_1,H_1\subset \Gamma_2$ would imply $\Gamma_1=L_1\oplus H_1\subset \Gamma_2$. Hence, by Lemma \ref{lema:planes}(b), $\Gamma_2\cap H_1$ is a  hyperplane of revolution of $E=\Gamma_2\cap K_1$.  

Next we look at 
$\Gamma_1\cap\Gamma_2$. This has codimension 1 in $\Gamma_2$. If it coincides with  $H_2$, then    $E=\Gamma_1\cap K_2=H_2\cap K_2$, which   is an ellipsoid, by Lemma \ref{lema:planes}(a). If $\Gamma_1\cap\Gamma_2\neq H_2$, then by Lemma \ref{lema:planes}(b), $\Gamma_1\cap H_2$ is a  hyperplane of revolution of $E=\Gamma_1\cap K_2$. 

Now $\Gamma_1\cap H_2$, 
$\Gamma_2\cap H_1$ are two distinct  hyperplanes of revolution of $E$, since $L_1$ is  contained in the first but not in the second. It follows from   Lemma \ref{lema:two}  that $E$  is an ellipsoid. 
\end{proof}

The  statement of the following lemma has appeared elsewhere 
(e.g., statement III of the proof of Theorem 2.2 of \cite{Mo2}), but we did not find a  published proof of it  (perhaps because it is intuitively clear and a hassle to prove).

\begin{lema} \label{lema:lata}  Let $B\subset\R^{n+1}$ be a symmetric convex body and $x_i\to x$ a convergent sequence in $ S^n$.   Assume  each hyperplane section $x_i^\perp\cap B$ is an affine symmetric body of revolution with an axis of revolution $L_i\subset x_i^\perp$. If  $\{L_i\}$ is a convergent sequence in $\R P^n$, $L_i\to L$, then $x^\perp\cap B$ is an  affine symmetric body of revolution with an axis of revolution $L$.
\end{lema}

\begin{proof}Let $\Gamma_i:=x_i^\perp, \Gamma:=x^\perp$, 
%
%
$K_i:=\Gamma_i\cap B, $ $K:=\Gamma\cap B$. Assume, without loss of generality, that $x=e_{n+1}$, so that $\Gamma=\R^n$. 

\mn{\bf Claim 1.} {\em $K_i\to K$ in the Hausdorf metric.}

 \mn 
 
 We postpone for the moment the proof  this claim (and the two subsequent ones). Define $\pi:\R^{n+1}\to\R^n$   by 
$(x_1,\ldots, x_{n+1})\mapsto (x_1,\ldots,x_n)$. Note that $\pi(K)=K$ and $\pi(L)=L$. 
 
\mn{\bf Claim 2.}  {\em For large enough
 $i$, $\pi|_{\Gamma_i}:\Gamma_i\to \R^n$  is a linear isomorphism.}

\mn 

%
%
We henceforth restrict to a subsequence of $\{K_i\}$ such that  each $\pi|_{\Gamma_i}$ is an isomorphism. Let $K_i':=\pi(K_i)\subset \R^n,$  $L_i':=\pi(L_i)\subset\R^n$.  Then each $K_i'\subset \R^n$ is an affine symmetric body of revolution with an axis of revolution $L_i'$, $L_i'\to L$ and $K_i'\to K$ (by Claim 1). 
%
%
By definition of affine symmetric body of revolution, there exist  linear isomorphisms  $T_i:\R^n\to \R^n$ such that $K_i'':=T_i(K_i')$ is a (honest) symmetric body of revolution. By postcomposing $T_i$ with appropriate elements of $GL_n(\R)$, we can also assume that $\R e_n=T_i(L_i')$ is an axis of revolution of $K_i''$,  that $\R^{n-1}\pm e_n$ are support hyperplanes of $K_i''$ at $\pm e_n$ and  that 
$K_i''\cap \R^{n-1}$ is the unit $n-1$ dimensional closed ball in $\R^{n-1}$,  centered at the origin.  

\mn{\bf Claim 3.}  
{\em $\{T_i\}$ is contained in a compact subset of $GL_n(\R)$.}

\mn 

It follows that there is a subsequence of $\{T_i\}$, which we rename $\{T_i\}$, converging to an element $T\in GL_n(\R)$.  Let $K'':=T(K).$ Then  $\lim K_i''=
\lim T_i(K'_i)=(\lim T_i)(\lim K'_i)=T(K)=K'', $ and $T(L)=(\lim T_i)(\lim L_i')=\lim T_i(L_i)=\R e_n$. It is thus enough to show that $\R e_n$ is an axis of revolution of $K''$. Now $\R e_n$ is an axis of revolution of each $K_i''$ hence $gK_i''=K_i''$ for all $g\in O_{n-1}$ (the elements of $O_n$ leaving  $\R e_n$ fixed). Taking the limit $i\to \infty$ we obtain $g(K'')=K''.$ Hence $\R e_n$ is an axis of revolution of $K''$.
%

\mn {\em Proof of  the 3 claims:}

\sn (1) 
%
%
Let $\Gamma\subset\R^n$ be a hyperplane and  $U\subset\R^n$ an open subset such that $\Gamma\cap B\subset U$. 
Then there is a $\delta>0$ such that $\Gamma_\delta\cap B\subset U$, where $\Gamma_\delta$ is the $\delta$-neighbourhood around $\Gamma$ (this follows since the distance between the compact $\Gamma\cap B$ and the closed $\R^{n+1}\setminus U$ is positive).

%
For $x,x'\in S^n$, let $\Gamma =x^\perp$ and $\Gamma'=x'^\perp$. For any fixed $R>0$, the ball of radius $R$ in $\Gamma'$ will be contained in $\Gamma_\delta$ provided $\Gamma$ and $\Gamma'$ are close enough 
%
%
(i.e., provided $\la x, x'\ra$ is close enough to 1). Thus $\Gamma'\cap B\subset \Gamma_\delta\cap B$ for $\Gamma$ and $\Gamma'$ sufficiently close. 

%
%
Fix an $\epsilon>0$ and take $U=K_\epsilon$; then there is $\delta >0$ such that $\Gamma_\delta\cap B\subset K_\epsilon$, but then $K_i=\Gamma_i\cap B\subset\Gamma_\delta\cap B\subset K_\epsilon$, for all $i$ sufficiently large.

The argument is symmetric, thus $K\subset (K_i)_\epsilon$ for all sufficiently large $i$.

\sn(2) $\mathrm{Ker}(\pi)=\R e_{n+1}$, hence $\mathrm{Ker}(\pi|_{\Gamma_i}) \neq 0$ if and only if
%
%
  $e_{n+1}\perp x_i$. But $x_i\to e_{n+1}$ implies $\la x_i, e_{n+1}\ra\to 1,$
hence $\la x_i, e_{n+1}\ra\neq 0$  for all $i$ sufficiently large. 

\sn (3) For each pair of constants $c,C>0$ the set of elements $A\in GL_n(\R)$ satisfying $c\|v\|\leq \|A v\|\leq C\|v\|$ for all $v\in\R^n$ is  clearly closed. It is also bounded because its elements satisfy $\|A\|\leq C$ (using the operator norm on $\mathrm{End}(\R^n)$).  It is thus enough to find constants $c,C>0$ such that  
%
%
$c\|v\|\leq \|T_i v\|\leq C\|v\|$ for all $v\in \R^n$ and all $i$. 

Denote by $B_\rho$ the closed ball in $ \R^n$ of radius $\rho$ centered at the 
origin. Then there are constants $r',R',r'',R''>0$ such that $B_{r'}\subset \pi(B)
\subset  B_{R'}$ and 
$B_{r''}\subset K_i''\subset  B_{R''}$ for all $i$. 
%
%
It follows that   $T_i(B_{r'})\subset T_i(K'_i)=K_i''\subset B_{R''},$ thus $\|T_iv\|\leq C\|v\|$ for all $v\in \R^n$ and all $i$, where $C=R''/r'.$

Next, $(T_i)^{-1} B_{r''}\subset(T_i)^{-1} ( K_i'')=K_i'\subset B_{R'}$, 
hence $\|(T_i)^{-1} w\|\leq c'\|w\|$ for all $w\in\R^n$ and all $i$, where $c'=R'/r''.$ Substituting $w=T_i v$ in the last inequality we obtain $c\| v\|\leq \|T_iv\|$ for all $v\in\R^n$ and all $i$,
%
%
 where $c=1/c'=r''/R'$. 
\end{proof}

\begin{lema}\label{lema:cont}
 Let $B\subset\R^{n+1}$ be a symmetric convex body, all of whose hyperplane sections are non-ellipsoidal affine symmetric bodies of revolution. For each $x\in S^n$ let $L_x$ be the (unique) axis of revolution of $x^\perp\cap B$. Then $x\mapsto L_x$ is a continuous function $S^n\to \R P^n$. 
\end{lema}

\begin{proof} Let $x_i\to x$ be a converging sequence in $S^n$.  To show that $L_{x_i}\to L_x$ it is enough to show that $L_{x_i}$ is convergent and its limit is an axis of revolution of $x^\perp\cap B$. Since $\R P^n$ is a compact metric space, to show that $L_{x_i}$ is convergent it is enough to show that all its convergent subsequences have the same limit. To show this, it is enough to show that the limit of a convergent subsequence of $L_{x_i}$ is an axis of revolution of $x^\perp\cap B$. This is the statement of  Lemma \ref{lema:lata}. 
\end{proof}

\subsection{The proof of Theorem \ref{thm:rev}} We first show Theorem \ref{thm:elip}, i.e., assume $B\subset \R^{n+1}$ is a symmetric convex body, all of whose hyperplane sections are affine symmetric bodies of revolution, and show that at least one of the hyperplane sections is  an ellipsoid.  If none of the sections  is an ellipsoid then, by Lemma  \ref{lema:two},   for each $x\in S^n$ the  section $x^\perp\cap B$  has a unique  axis of revolution $L_x\subset x^\perp.$
By Lemma \ref{lema:cont}, $x\mapsto L_x$ defines a continuous function $S^n\to \R P^n$, i.e., a line subbundle of $TS^n$.  (Note that for even $n$ this is already a contradiction, so we proceed for odd $n$.)  Now every line bundle on $S^n$, $n\geq 2$,  is trivial, i.e., admits a non-vanishing section, hence one can find a continuous function $\psi:S^n\to S^n$ such that $\psi(x)\in L_x$ for all $x\in S^n$. Since $\psi(x)\perp x$, the function   $F(t,x):=(t\psi(x)+(1-t)x)/\|t\psi(x)+(1-t)x\|$, $0\leq t\leq 1$, is well defined (the denominator does not vanish), defining a homotopy between $\psi=F(1,\cdot)$ and the identity map $F(0,\cdot)$. It follows that $\psi$ is a degree 1 map and is thus {\em surjective}.

  Now let $\Gamma_2\cap B$ be a hyperplane section of $B$, with  hyperplane of revolution $H_2\subset \Gamma_2$. Let $L_1\subset H_2$ be any 1-dimensional subspace. Then the surjectivity of $\psi$ implies that $B$ admits a  hyperplane section $K_1=\Gamma_1\cap B$ with axis of revolution $L_1$.  By Lemma \ref{lema:k1}, $K_1$ is an ellipsoid, in contradiction to our assumption that none of the hyperplane sections of $B$ is an ellipsoid.   This completes the proof of Theorem  \ref{thm:elip}. 
  
To complete the proof of  Theorem \ref{thm:rev}, we use Theorem  \ref{thm:elip} to conclude that all hyperplane sections of $B$ are ellipsoids, and hence that $B$ itself is an ellipsoid, as needed. \qed

\begin{rmk}\label{rmk:conv}
Lemma \ref{lema:planes} says that any hyperplane section of an affine symmetric convex body of revolution $B$ is again an affine symmetric convex body of revolution. The converse of this result, as far as we know, is an open problem. Let us state a somewhat more general question:

\sn

{\em Let $B\subset\R^{n+1}$, $n\geq 4$, be a convex body containing the origin in its interior. If every hyperplane section of B  is an affine body of revolution, is $B$  necessarily an affine body of revolution?}

\sn 

An obvious necessary condition for $B$ to be an affine body of revolution is that one of its hyperplane sections is an  ellipsoid (take the hyperplane of revolution of $B$). Thus, Theorem \ref{thm:elip} can be viewed as a first step for a positive answer to the above question (at least, under the further assumption of symmetry). Since Theorem \ref{thm:elip} assumes $n\geq 4$, we dare only ask the above question under the same dimension restriction.

The case $n=2$ has a different flavour altogether, where `axis of revolution' of a plane section is replaced by `axis of symmetry'. (For example, there are convex plane regions with several different axes of symmetry which are not ellipses; this is the reason we proved  Theorem \ref{thm:elip} only for $n\geq 4$). Yet there is a result in this dimension, somewhat related to Theorem \ref{thm:elip}. It is Theorem 2.1 of \cite{Mo2}: {\em Let $B\subset\R^3$ be a convex body such that every plane section through some fixed interior point of $B$ has an axis of symmetry. Then at least one of the sections is a disk.}
\end{rmk}

\section{Structure groups of spheres}
\subsection{A reminder on structure groups of manifolds and their reduction}\label{sec:str}

First, let us recall the following basic definitions (see, for example, \S 5 of Chap.~I of \cite{KN}, or Part I of \cite{St}). 

Let $G$ be a  topological group,  $M$ a topological space and $P\to M$ a principal $G$-bundle. A {\em reduction of the structure group} of   $P\to M$  to a closed subgroup $H\subset G$ is a principal  $H$-subbundle of $P$. 
Equivalently, it is a continuous section of the bundle $P/H\to M$ associated with the left $G$-action on $G/H$.  
The {\em frame bundle} of an $n$-dimensional  differentiable manifold $M$ is the 
$GL_n(\R)$-principal  bundle $ F(M)\to M$, whose fiber at a point $x\in M$ is the set of all linear isomorphisms $\R^n\to T_xM$, with the $GL_n(\R)$ right action given by precomposition of linear maps. 
A {\em $G$-reduction} of the structure group of  a smooth $n$-manifold $M$ (or a {\em $G$-structure}) is the reduction  of the structure group $GL_n(\R)$ of its frame bundle to a  closed subgroup    $G\subset GL_n(\R)$. Equivalently, it  is given by an open  cover of $M$, together with a trivialization of the restriction of $TM$ to each of the covering open subsets, such that the transition functions between the trivializations on overlapping members of the cover take values in $G$ (Prop.~5.3 of \cite{KN}, p.~53).  






\begin{rmk}\label{lema:con}
For $M=S^n,$ there is a standard cover by two `hemispheres', intersecting along a neighborhood of the `equator' $S^{n-1}$, hence its structure group is given by a single transition function $\chi_n:S^{n-1}\to GL_n(\R)$, called the {\em characteristic map} (\S18 of \cite{St}, pp.~96-100). The structure group of  $S^n$  can be reduced  to $G$ if and only if the characteristic map $\chi_n$ is homotopic to a map whose image is contained in $G$. In particular, since $S^{n-1}$ is connected ($n\geq 2$),
if the structure group of $S^n$ can be reduced  to some closed subgroup  $G\subset GL_n(\R)$  then it can be further reduced to its  identity component $G^0\subset G$. 
\end{rmk}
Let us recall   Lemma \ref{lema:key}, announced in the introduction. It follows from Lemma 2 of \cite{G}, but since it is such a key result in this article, we offer here an alternative proof, somewhat more elementary and detailed. 

\mn{\bf Lemma \ref{lema:key}.} {\em Let $B\subset \R^{n+1}$ be a symmetric convex body, all of whose hyperplane sections are linearly equivalent to some fixed symmetric convex body $K\subset \R^n$. Let $G_K:=\{g\in GL_n(\R)| g(K)=K\}$ be the   group of linear symmetries of $K$. Then the structure group of  $S^n$ can be  reduced to  $G_K$. }

\begin{proof} Identify for each $x\in S^n$, by parallel translation in $\R^{n+1}$, the tangent space to $S^n$ at $x$ with $x^\perp\subset\R^{n+1}$ and define the set $P_x\subset F_x(S^n)$  of frames at $x$  as  the set of linear isomorphisms  $\R^n\to x^\perp$ mapping  $K$ to $x^\perp\cap B$. Note that if $\phi\in P_x$ then $P_x=\phi \,G_K$, that is, $\sigma:x\mapsto P_x$ is a section  of $F/G_K\to S^n$. In order to show that $P=\bigcup P_x\subset F$ is a $G_K$-reduction it is thus enough to show that (1) $G_K$ is a closed subgroup of $GL_n(\R)$ and (2)  $\sigma:S^n\to F/G_K$ is continuous (see, e.g., \cite{Hus}, Theorem 2.3, p.~74, or \cite{St}, Corollary 9.5, p.~43).  By Lemma \ref{lema:bdd} above, $G_K$ is a compact group, hence it is closed in $GL_n(\R)$.  

To prove the continuity of $\sigma$, it is enough to show that for every convergent sequence  $x_i\to x$  in $S^n$    there exists a subsequence of $\{\sigma(x_i)\}$ converging  to  $\sigma(x)$. 
Let $\pi:F(S^n)\to F(S^n)/G_K$ be the natural projection and choose arbitrary  lifts  $\phi_i\in P_{x_i}$ of $\sigma(x_i)$. By the continuity of $\pi$, it is enough to find a subsequence of $\{\phi_i\}$ converging to an element  $\phi\in P_x$.


Now each $\phi_i$ is a linear isomorphism $\R^n\to x_i^\perp\subset\R^{n+1}$, thus we may think of $\phi_i\in \mbox{Hom}(\R^n,\R^{n+1})$.
Since $\phi_i(K)\subset B$, with int$(K)\neq\emptyset$ and $B$ compact, and hence bounded,
 $\{\phi_i\}$ is a bounded set in $\mbox{Hom}(\R^n,\R^{n+1})$ ($K$ contains some basis $\beta$ of $\R^n$ and $\phi_i(\beta)\subset B$). Therefore, $\{\phi_i\}$ has a convergent subsequence which we denote by $\phi_i$ as well, $\phi_i\to \phi$, for some $\phi \in\mbox{Hom}(\R^n,\R^{n+1})$. It remains to show that $\phi\in P_x$, i.e., $\phi$ is a linear isomorphism $\R^n\to x^\perp$ such that  $\phi(K)=x^\perp\cap  B$. 

Let $K_i=x_i^\perp\cap B$, $K_\infty=x^\perp\cap B$. In the proof of  Lemma \ref{lema:lata} (claim 1) we showed that $x_i\to x$ implies $K_i\to K_\infty$ (in the Hausdorff metric). Thus, $\phi(K)=(\lim\phi_i)(K)=\lim(\phi_i(K))=\lim K_i=K_\infty.$ 
Since $K_\infty$ has non empty interior in $x^\perp$, $\phi(K)=K_\infty$ implies that $\phi$ is a linear isomorphism $\R^n\to x^\perp.$ Thus $\phi\in P_x$, as needed. 
\end{proof}

\subsection{Proof of Theorem \ref{thm:red}a (the reducible case).}\label{sec:red}
Suppose the structure group of $S^n$ can be reduced to a closed connected subgroup $G\subset SO_{n-1}$, acting reducibly on $\R^n$. Then $G$ is conjugate  to a closed connected subgroup 
 $G'\subset SO_k\times SO'_{n-k}\subset SO_n$ for some $k$, $n/2\leq k<n$, where $SO'_{n-k}$ denotes the subgroup of $SO_n$
   fixing $\R^{k}=\{ x_{k+1}=\ldots=x_n=0\}\subset\R^n$.  
   If $n\equiv 1$ mod 4, then such a reduction is possible only if  $k=n-1$, i.e., $G'\subset  SO_{n-1}$, 
   acting irreducibly on $\R^{n-1}$ (see  \cite{St}, \S27.14, \S27.18, 
   pp.~143-144). In particular, the structure group of $S^n$ reduces to $SO_{n-1}$ 
   but not to $SO_{n-2}$. 
We shall next prove  that $G'$ acts transitively on $S^{n-2}$ (see Corollary 3.2 of \cite{L}). 
   
Consider the standard fibration $SO_{n-2}\to SO_{n-1} \stackrel{\pi}{\to} S^{n-2}$. If $G'$ does not act transitively on $S^{n-2}$ it means that the composition 
$G'\stackrel{i}{\hookrightarrow} SO_{n-1}\stackrel{\pi}{\to} S^{n-2}$ is not surjective, and is therefore null homotopic.  Let $F:G'\times I\to   S^{n-2}$  be the homotopy.  Then, by the homotopy lifting property, there exists a map  $\widetilde{F}$  completing the diagram
$$
\begin{tikzcd}[row sep=large]
G'\arrow[swap,"I\times 0",d] \arrow[r, hook,"i"] & SO_{n-1} \arrow[d, "\pi"]\\
G'\times I\arrow[ur,"\widetilde F"] \arrow[swap,r,"F"] &S^{n-2}
\end{tikzcd}
$$
Commutativity of the diagram implies  that $\widetilde{F}(x,1)\in SO_{n-2}\subset SO_{n-1}$ for every $x\in G'$. Let $f:G'\to SO_{n-2}$ be defined by $f(x)=\widetilde{F}(x,1)$; then, up to homotopy, the following diagram commutes
$$
\begin{tikzcd}[column sep=tiny, row sep=large]
G'\arrow[dr,"f",swap] \arrow[rr, hook,"i"] & &SO_{n-1} \\
&SO_{n-2}\arrow[ur,"j",swap]&
\end{tikzcd}
$$%
But now, precomposing $j\circ f$ with the characteristic map $\chi_n: S^{n-1}\to G'$, yields a reduction of the structure group of $S^n$ to $SO_{n-2}$, which is a contradiction.
$$
\begin{tikzcd}[column sep=small, row sep=large]
S^{n-1} \arrow[r,"\chi_n"]&G'\arrow[dr,"f",swap] \arrow[rr, hook,"i"] & &SO_{n-1} \\
&&SO_{n-2}\arrow[ur,"j",swap]&
\end{tikzcd}
$$
\qed

\subsection{Proof of Theorem \ref{thm:red}b (the irreducible case)}\label{sec:irred}
We start  with the following three preliminary lemmas. 

\begin{lema}\label{lema:A}
{\it For all  $n\equiv 1$ mod 4,  $n\geq 5$,  if the structure group of $S^n$ can be  reduced to $G\subset SO_n$, then  $\dim G\geq n-2.$}
\end{lema}

\begin{proof}  This follows readily from Proposition 3.1 of \cite{CC}, since -- as mentioned above -- the structure group of $S^n$, $n\equiv 1\mbox{ mod } 4$, may be reduced to $SO_{n-1}$ but not to $SO_{n-2}$. Given that the argument is a simple one, we include it here.

Assume that $\dim G=k<n$. We are going to show that the structure group of $S^n$ reduces to the standard $SO_{k+1}\subset SO_n$. This implies the result.

Consider the characteristic map $\chi_n :S^{n-1}\to SO_n$ of $S^n$. Assuming that the structure group of $S^n$  reduces to $G$ amounts to the existence of $f:S^{n-1}\to G$ such that the following diagram commutes up to homotopy:
$$
\begin{tikzcd}[column sep=small]
S^{n-1}\arrow[dr,"f",swap] \arrow[rr, "\chi_n"] & &SO_n \\
&G\arrow[ur,"i",swap]&
\end{tikzcd}
$$
The standard inclusion $SO_{k+1}\hookrightarrow SO_n$ induces isomorphisms $\pi_j(SO_{k+1})\simeq\pi_j(SO_n)$ for every $j<k$ (this follows immediately from the long exact sequences of the fibrations $SO_{k+1+r}\to SO_{k+2+r}\to S^{k+1+r}$ for the range of $j$'s in question).

Now, this implies that $G\hookrightarrow SO_n$ factors (up to homotopy) through $SO_{k+1}$. One way of seeing this is via obstruction theory. Think of $G$ as a CW-complex. Then the obstruction to extend the inclusion $G\hookrightarrow SO_{k+1}$ from the $j$-skeleton to the $j+1$-skeleton is a cocycle with coefficients in $\pi_j(SO_{k+1})$. But the inclusion $SO_{k+1}\hookrightarrow SO_n$ induces isomorphisms onto $\pi_j(SO_n)$ ($j<k$) where we know that the obstruction vanishes. Therefore, there is no obstruction to construct $G\to SO_{k+1}$ such that $G\to SO_{k+1}\hookrightarrow SO_n$ is homotopic to the inclusion $G\hookrightarrow SO_n$.
Hence, the structure group of $S^n$ reduces to $SO_{k+1}.$
\end{proof} 

\begin{lema}\label{lema:B} {\it If $n\geq 8$,  then the  structure group of  $S^n$ cannot be  reduced  to an irreducible subgroup $G\subsetneq SO_n$   isomorphic to   $SO_k, SU_m$ or $Sp_m$, with  $k\geq 4, m\geq 2$.}
\end{lema}

\begin{proof}  This is Corollary 2.2 of \cite{CC}.
\end{proof}

\begin{lema}\label{lema:C}  For all $n\geq 2$, if the structure group of $S^n$ reduces to a closed connected irreducible {\em maximal} subgroup $H\subsetneq SO_n$, then $H$ is {\em simple}.
\end{lema}

\begin{proof}  See Theorem 3 of \cite{L}. \end{proof}

We now proceed to the proof of Theorem \ref{thm:red}b, using the above three lemmas. We first treat  $n\geq 9$, then  $n=5$.

\mn{\bf The case $n\geq 9$.} Assume that $G\subset SO_n$ acts irreducibly on $\R^n$ but is not all of $SO_n$. Then it is contained in  some {\em maximal} connected closed subgroup $H$,  $G\subset H\subsetneq SO_n$. The structure group of $S^n$ then reduces to $H$, acting  also irreducibly  on $\R^n$. By  Lemma \ref{lema:C}, $H$ is  simple. By Lemma \ref{lema:B}, $H$ is  a non-classical group, i.e., it is isomorphic to either $Spin_m$, $m\geq 7,$ or one of the 5 exceptional simple Lie groups: $G_2$, $F_4$, $E_6$, $E_7$ or $E_8$. By Lemma \ref{lema:A}, $n\leq \dim H+2$.
Let $V$ be the complexification of the (irreducible) representation of $H$ on $\R^n$. Since $\dim V$ is odd, $V$  is a complex irreducible representation. 

Let us list all the properties of the pair $(H,V)$ that we have so far:
\begin{enumerate}[(i)]
\item $H$ is a non-classical compact connected group, i.e.,  $Spin_m$, $m\geq 7$, or one of the  five exceptional compact simple Lie groups. 

\item $V$ is a complex irreducible representation of $H$ of {\em real type} (i.e., the complexification of a real irreducible representation). 

\item $\dim V\equiv 1$ mod 4.

\item $\dim V \leq \dim(H)+ 2.$

\item If $H=Spin_m$, then its action on $V$ does not factor through $SO_m$. 

\end{enumerate}

We claim that these 5 conditions on the pair $(H,V)$ are {\em incompatible}, for $\dim V\geq 9$, except if $V$ is the complexified adjoint representation of $H=E_7$, in which case  $\dim V=\dim H=133\equiv 1$ mod 4. We are unable to exclude this case. 

\sn

For the exceptional groups, one can simply check (e.g., in Wikipedia) that none of them, other than $E_7$,  has a non-trivial  irreducible representation satisfying conditions (iii) and (iv). In the following table  we list the smallest irreducible representations for them; we have marked in boldface the first dimensions that are $\equiv 1$ (mod 4).

\sn
\begin{center}
\begin{tabular}{|l|c|c|c|c|c|}
 Group&  $G_2$&$F_4$  &$E_6$  &$E_7$& $E_8$  \\\hline
 $\dim H$&14  &52  &78  &133 & 248  \\\hline
\mbox{Irreps}& 7&26  &27  & 56 & 248 \\
 &14&52  & 78 &\fbox{{\bf 133}}  & 3875\\
 &27&{\bf 273}&351&912&\vdots \\
 &64& \vdots & {\bf 2925} & \vdots & {\bf 1763125}\\
 &{\bf 77}&\vdots &\vdots &\vdots &\vdots
\end{tabular}
\end{center}

\bn

For the spin groups, the next lemma  shows that conditions (iii) and (v) are incompatible. (We thank Ilia Smilga for kindly informing us about this lemma and its proof). 

 \begin{lema}\label{lema:ilia}Every   irreducible complex representation of $Spin_m$, $m\geq 3$,   which does not factor through $SO_m$ is even dimensional. 
 \end{lema}
 
 \begin{proof} We first review some well-known general facts concerning representations of simple compact Lie groups (see, for example, \cite{A}). With each $d$-dimensional complex representation of a compact semi-simple Lie group $G$ of rank $r$  with a maximal torus $T$, one can associate its  weight system $\Omega\subset \t^*$, a subset with $d$ points (counting multiplicity). The Weyl group $W=N_G(T)/T$ acts on $\t^*$, preserving $\Omega$. Thus, to show that $d$ is even, it is enough to show the following:
 
\begin{enumerate}[\quad(a)] 
\item An irreducible non classical representation $V$ of $Spin_m$ does not have a 0 weight.
 
\item  The Weyl group of  $Spin_m$ contains a subgroup whose order is a positive power of 2, and whose only fixed point in $\t^*$ is 0. 
\end{enumerate}

Note that (a) and (b) imply that $d$  is even, since under the action of  said subgroup of $W$, say $W'$, $\Omega$ breaks into the disjoint union of $W'$-orbits, each with an even number of elements, since, by (a),  all stabilizers  are strict subgroups of  $W'$, hence  have  even index. 

\sn

To show (a), note that the  $T$ action on the 0 weight space is  trivial. Now $-1\in Spin_m$ is in $T$  (since it is central), but $-1$ must act on $V$ by $-Id$, else the $Spin_m$ action on $V$ would factor through $SO_m=Spin_m/\{\pm 1\}$.

To show (b), let us first take  $m=2k.$   Then $\R^m$ decomposes under $T$  as the direct sum of $k$ $2$-planes.  Consider the subgroup  $N'\subset SO_m$ which leaves invariant each of these 2-planes. Then $N'\simeq S(O_2\times \ldots \times O_2)$, $T\subset N' \subset N(T)$, and its image  $W'=N'/T\subset W=N(T)/T$  acts on $\t^*$ by  diagonal matrices  with entries $\pm 1$ on the diagonal, with an even number of $-1$'s.   Using this description, it is easy to show that  $W'$  has order $2^{k-1}$ and that its only  fixed point in $\t^*$ is  $0. $

For  $m=2k+1$ the argument is simpler.  Under  $T$,  $\R^m$ decomposes as a direct sum of $k$  2-planes, plus a line. We take an element in $ SO_m$  which is a reflection about a line through the origin in each of these planes, and $ (-1)^k $ in the  line. This is in $N(T)$ and acts on $\t^*$ by $-Id$, hence its image in $W$ has order 2 and its only fixed point in $\t^*$ is the origin.
\end{proof}

\n{\bf The case $n=5$.}
The only reduction  of the structure group of $S^5$ that cannot be ruled out by Lemmas \ref{lema:A},  \ref{lema:B} or \ref{lema:C} is the 5-dimensional  irreducible representation of $SO_3$. This case  is eliminated   by the next lemma. 

\begin{lema}\label{lema:D}Let  $\rho:SO_3\to SO_5$ be the irreducible 5 dimensional representation of $SO_3$. Then, for any $f:S^4\to SO_3$, the composition  $S^4\stackrel{f}{\to} SO_3\stackrel{\rho}{\to} SO_5$ is null homotopic. It follows that  the  structure group of  $S^5$ cannot be  reduced  to $\rho$.
\end{lema}

\begin{proof} Since  the tangent bundle of $S^5$ is not trivial, the characteristic map $\chi_5:S^4\to SO_5$  is not null-homotopic. Consequently, to show that the structure group of $S^5$ cannot be reduced to $\rho$ it is enough to show that  any  composition  $S^4\stackrel{f}{\to} SO_3\stackrel{\rho}{\to} SO_5$ is null homotopic. To show this, we use the following three claims.

 \begin{enumerate}[ \quad(a)]
 \item $\pi_3(S^3)\simeq\pi_3(SO_3)\simeq \pi_3(SO_5)\simeq \Z$, $\pi_4(S^3)\simeq\pi_4(SO_3)\simeq \pi_4(SO_5)\simeq \Z_2$. 
 
 \item The map $\rho_*: \pi_3(SO_3)\to \pi_3(SO_5)$  has a  cyclic cokernel  of  {\em  even} order (the `Dynkin index' of $\rho$).  
 
 \item For any topological group $G$ and integers $k,n\geq 2$, the composition of maps $S^n\to S^k\to G$ defines a {\em bi-additive} map  $\pi_k(G)\times \pi_n(S^k)\to \pi_n(G)$, $([f],[g])\mapsto [f]\circ [g]:=[f\circ g]$ (the `composition product'). 
 \end{enumerate}

Claim (a) is standard (see, e.g., \cite{I}, Vol.~2, App.~A, Table 6.VII, p.~1745). Claim (b) is a straightforward Lie algebraic calculation, see next subsection. For claim (c), see \cite{W}, Theorem (8.3), p.~ 479.

%
%
 Now let $f:S^4\to SO_3$ be any (pointed) continuous map and $\tilde f:S^4\to S^3$ its lift to the universal double cover $\pi:S^3\to SO_3$.  By (b),  the  composition  $\tilde \rho:=\rho\circ\pi:S^3\to SO_5$ has  an even Dynkin index (in fact, it is the same as the index of $\rho$, since $\pi$, being a cover, has index 1). 
In particular, $[\tilde\rho]=2[u]\in \pi_3(SO_5)$, for  some $u:S^3\to SO_5$. By (c), with $n=4, k=3, G=SO_5$, $[ \rho\circ f]=[\tilde \rho\circ \tilde f]=[\tilde \rho]\circ [\tilde f]=(2[u])\circ [\tilde f]=2([u]\circ [\tilde f])=0\in\pi_4(SO_5)\simeq\Z_2.$
$$\begin{tikzcd}
& S^3 \arrow{d}{\pi}\arrow{dr}{\tilde \rho}&\\
S^4\arrow{ur}{\tilde f} \arrow[swap]{r}{f} & SO_3 \arrow[swap]{r}{\rho}& SO_5
\end{tikzcd}
$$
\end{proof} 

A byproduct of the proof of Theorem \ref{thm:red} is the following corollary that could be of some interest to topologists.

\begin{cor}\label{cor:last}Suppose that the structure group of $S^n$ can be
 reduced to a closed connected subgroup $G\subsetneq SO_n$. If 
$n=4k+1\geq 5$,  but  $n\neq 9, 17$ or $133$, then $G$ is conjugate to  
the standard inclusion of  $SO_{4k}$, $U_{2k}$ or $SU_{2k}$ in  $SO_{4k+1}.$ 
For $n=9$, $G$ is conjugate to the standard inclusion of 
$SO_{8}$, $U_{4}$, $SU_{4}$ or  $Spin_7\subset SO_{8}$ in $SO_9$. 
\end{cor}
\begin{proof}  By Theorem \ref{thm:red}(b),  such a $G$ is conjugate to a subgroup of the standard inclusion  $SO_{4k}\subset SO_{4k+1}$, acting transitively on  $S^{4k-1}.$
The only closed connected subgroups $G\subset SO_{4k}$  acting  transitively on $S^{4k-1}$, in the said dimensions, are the standard linear actions  of   $ SO_{4k},$ $ U_{2k}, SU_{2k}, Sp_kSp_1, Sp_kU_1, Sp_k$ on $\R^{4k}=\C^{2k}=\H^k$, or the spin representation of $Spin_7$ on $\C^4$ (see, e.g., \cite[7.13, p.~179]{Be}). But the groups $Sp_kSp_1, Sp_kU_1, Sp_k$, $k\geq 1$, cannot occur  as structure groups of $S^{4k+1}$, since they contain  the last one, $Sp_k$, which is excluded by Theorem 2.1 of \cite{CC}.
\end{proof} 

\begin{rmk} For $n=17$, the group $Spin_9\subset SO_{16}$ acts transitively on $S^{15}$, but we do not know if the structure group of $S^{17}$ could be reduced to it. For $n=133$, as explained before, we do not know if the group $E_7\subset SO_{133}$ (or some subgroup of it acting irreducibly on $\R^{133}$) may appear as a reduction of the structure group of $S^{133}$.
\end{rmk}
\subsection{The Dynkin index}
Here we prove claim (b) from the proof of Lemma \ref{lema:D} of the previous subsection.  We begin with some background. 

Let $\rho:H\to G$ be a homomorphism of compact simple Lie groups. The third homotopy group of any simple Lie group is infinite cyclic (isomorphic to $\Z$),  hence the induced map $\rho_*:\pi_3(H)\to\pi_3(G)$   has 
a cyclic cokernel of order  $j\in\N$,  called the {\em Dynkin index} of $\rho$ (if $\rho_*=0$ then $j=0$, by definition). Clearly, $j$ is {\em multiplicative}, i.e., if  $\tilde H$ is a simple compact Lie group and $\pi:\tilde H\to H$ is a homomorphism, then $j( \rho\circ\pi)= j(\rho)j(\pi).$

There is a simple Lie algebraic expression for $j(\rho)$. To state it, the  Killing form on any simple compact Lie algebra needs to be normalized first  by  $\la\delta,\delta\ra=2$, where   $\delta$ is the longest root. Next, the pullback by $\rho:H\to G$ of the Killing form of $G$   is an $\Ad_H$-invariant quadratic form on the Lie algebra of $H$, hence, by simplicity of $H$,  is a non-negative multiple of the  Killing form of $H$. This multiple turns out to be precisely the Dynkin index of $\rho$.

\begin{teo}\label{thm:dyn} Let $\rho:H\to G$ be a homomorphism of
 compact simple Lie groups and $\rho_*:\h\to\g$ the induced Lie algebra homomorphism. Then
\begin{equation}\label{eq:ind}
\la \rho_*X, \rho_*Y\ra_\g=j(\rho)\la X, Y\ra_\h
\end{equation}
for all  $X,Y\in \h$.
\end{teo}

In fact, Dynkin {\em defined}  $j(\rho)$ via Formula \ref{eq:ind} (see \cite[formula (2.2), p.~130]{Dy}), and showed in the same article that $j(\rho)$ is an integer, without reference to its topological interpretation. Later, it was shown  to have an equivalent definition via homotopy groups, as given above  (we are not sure who proved it first, we learned it from  \cite{On}, \S 2 of Chapter 5, p.~257).

\begin{lema} $j(\rho)=10$ for the irreducible representation $\rho:SO_3\to SO_5$.
\end{lema}
\begin{proof} Theorem \ref{thm:dyn} gives an easy to follow  recipe for $j$. To apply it, one needs to  compute first the normalization of the Killing forms of $SO_3$ and $SO_5$. 

 Let $\so_5$  be the set of $5 \times 5$ antisymmetric real matrices, the Lie algebra of $SO_5,$  with $\t\subset\so_5 $ the set  of block diagonal
matrices of the form $\left( x_1J\oplus  x_2J\oplus 0\right)$, where
$J=\left(\begin{smallmatrix}0&-1\\ 1&0\end{smallmatrix}\right)$. The roots are 
$\pm x_1\pm x_2, \pm x_1,\pm x_2$, with  $\delta:=x_1+x_2$. 
Since $\tr(XY)$ is clearly an $\Ad$-invariant non-trivial bilinear form on $\so_5$, the normalized Killing form of $\so_5$ is of the form $\la X, Y\ra=\lambda\, \tr(XY),$ for some $\lambda\in \R$. The normalization condition is $\la\delta^\flat,\delta^\flat\ra=2,$ where $\delta^\flat\in\t$ is defined via $\delta(X)=\la\delta^\flat, X\ra$ for all $X\in \t$. Let $\delta^\flat=\lambda'(J\oplus J\oplus 0),$ for some $\lambda'\in \R$. Then for all $X \in \t$, $\la \delta^\flat,X\ra
 =\lambda\tr(\delta^\flat X)=
 -2\lambda\lambda' \delta(X)$, thus $-2\lambda\lambda'=1,$ 
 so $\delta^\flat=-{1\over 2\lambda}(J\oplus J\oplus 0)$ and 
 $2=\la \delta^\flat, \delta^\flat\ra=\lambda\tr[(\delta^\flat)^2]=-1/\lambda$, hence $\lambda=-1/2$. It follows that  
 $\la X,Y\ra_{\mathfrak{so}_5} =-\tr(XY)/2.$ 
For $\so_3$ we get by a similar argument  
 $\la X,Y\ra_{\mathfrak{so}_3} =-\tr(XY)/4.$ 
 
Now let $\rho:SO_3\to SO_5$ be the 5-dimensional irreducible representation on $\R^5$ (conjugation of traceless symmetric $3\times 3$ matrices).  Let $X=(J\oplus 0)\in\so_3$.  To calculate $\tr[(\rho_*X)^2]$,  we let $X$ act on $S^2((\C^3)^*)$ (complexifying, passing to the dual and adding an extra trivial summand does not affect trace). Now $x_1\pm ix_2,  x_3$ are  $X$ eigenvectors in $(\C^3)^*$, with eigenvalue $\pm i, 0$, hence  the  eigenvalues of the $\rho_*X$ action on $S^2((\C^3)^*)$ are  $\pm 2i, \pm i,0,0$, and those of $(\rho_*X)^2
$ are $-4,-4,-1,-1,0,0$, giving $\tr[(\rho_*X)^2]=-10.$ Thus $j(\rho)=\la \rho_*X,\rho_*X\ra_{\so_5}/\la X,X\ra_{\so_3}=2\,\tr[(\rho_*X)^2]/\tr(X^2)=10,$ as claimed. 
\end{proof}

\end{document}